%% file: guillermouc3mnew.tex
\documentclass[graybox]{svmult}

\usepackage{type1cm}        
\usepackage{makeidx}         
\usepackage{graphicx}        
\usepackage{multicol}        
\usepackage[bottom]{footmisc}

\usepackage{newtxtext}       %
\usepackage{newtxmath}       

\makeindex             

\begin{document}

\title*{Convergent non complete interpolatory quadrature rules}
\author{U. Fidalgo and J. Olson}
\institute{Ulises Fidalgo \at Department of Mathematics, Applied Mathematics, and Statistics,
Case Western Reserve University, Cleveland, Ohio 43403, \email{uxf6@case.edu} This author's research was supported in part by the research grant MTM2012-36372-C03-01 funded by "Ministerio de Econom\'ia y Competitividad", Spain.
\and Jacob Olson \at Case Western Reserve University, Cleveland, Ohio 43403  \email{jfo27@case.edu}}

%
%

\maketitle

\abstract*{We find a family of convergent schemes of nodes for non-complete interpolatory quadrature rules.}

\abstract{We find a family of convergent schemes of nodes for non-complete interpolatory quadrature rules.}

\section{Introduction}
\label{sec:1}
Let $\mathcal{C}([-1,\,1])$ be the set of all continuous functions defined on $[-1,\,1]$. Given an $n$-tuple of nodes  ${\bf x}_n=(x_{1,n},\ldots,x_{j,n})$ satisfying   $-1<x_{1,n}<x_{2,n}<\cdots<x_{n,n}<1$, we consider {\it integration rules}  
\begin{equation}\label{rule}
I_n[f]= \sum_{j=1}^n w_{j,n} f(x_{j,n}), \quad f \in \mathcal{C}[-1,\,1]
\end{equation}
associated to the integrals 
\begin{equation}\label{integrals}
\displaystyle I(f)=\int_{-1}^1 f(x)\, \mbox{d} \lambda_0(x), \quad \mbox{where} \quad  \frac{\mbox{d} \lambda_0(x)}{\mbox{d} x}=\frac{1}{\pi \sqrt{1-x^2}}.
\end{equation}
The numbers $w_{j,n},$ $j=1,\ldots,n$ are called weights.

An integration rule $I_n[\cdot]$ is said to be interpolatory if there exists a number $m \in \left\{0,1,\ldots, 2n-1\right\}$, such that the following equality holds for every polynomial $p$ with degree $\leq m$ (we denote $p \in \Pi_m$): 
\begin{equation}\label{interpolatoryq}
I_n[p]=\int_{-1}^1 p(x)\, \mbox{d} \lambda_0(x).
\end{equation}
When the equality (\ref{interpolatoryq}) holds for certain $m$, and is not extendable for all polynomials with degree $m+1$, we say that $I_n[\cdot]$ is an interpolatory quadrature rule with $m$-degree of exactness. When $m=2n-1$ $I_n$ is the Gaussian quadrature rule.

Consider a sequence of interpolatory quadratures $\displaystyle \left\{I_n\right\}_{n \in \Lambda}$ constructed with the following schemes of nodes and weights 
\begin{equation}\label{nodessequence}
  {\bf x}=\left\{{\bf x}_n=(x_{1,n},\ldots,x_{n,n})\right\}_{n\in \Lambda} \quad \mbox{and} \quad {\bf w}=\left\{{\bf w}_n=(w_{1,n},\ldots,w_{n,n})\right\}_{n\in \Lambda},
 \end{equation}
respectively. 
 
We say that  $\displaystyle \left\{I_n\right\}_{n \in \Lambda}$ is convergent if
\begin{equation}\label{conv}
\lim_{n\in\Lambda} I_n[f]=\int_{-1}^1 f(x)\, \mbox{d} \lambda_0(x),\quad \mbox{for all} \quad f\in C[-1,1].
\end{equation} 
According to a classical result of P\'olya \cite[page 130]{DR}, when $\displaystyle  m(n) \to \infty$ as $n\to \infty $, the equality \eqref{conv} holds true if and only if $ \displaystyle \sup_{n\in \Lambda} \sum_{j=1}^n \left|w_{j,n}\right| < \infty.$
This condition is satisfied if the weights $w_{j,n}$ are all positive. From (\ref{interpolatoryq}) we observe that 
\begin{equation}\label{polyacond}
\sum_{j=1}^n w_{j,n}=\int_{-1}^1\, \frac{\mbox{d} x}{\pi \,\sqrt{1-x^2}}=1 < \infty.
\end{equation}
In the Gaussian quadrature rule (maximum degree of exactness  $m(n)=2n-1$) the weights $w_{j,n}$, $j=1,\ldots, n$ are all positive and the convergence of the rule is guarantied. However the nodes are all fixed. For each $n \in \mathbb{N}$ the points of evaluation $x_{j,n}$, $j=1,\ldots,n$ must be the roots of the $n$th orthogonal polynomial with respect to $\lambda_0$ (see for instance \cite{Sze}). This is the Chebyshev polynomial with degree $n$. This means that if we do not have  the value of $f$ at each point $x_{j,n}$ the calculus gets stuck. It is convenient to have more flexibility in the distribution of the evaluation nodes. We study convergent interpolatory integration rules with orders of exactness $m <2n-1$. 

The authors of \cite{BLS2} analyze a wide class of interpolatory quadrature rules with $m(n)$ degrees of exactness behaving as follows
\begin{equation}\label{limita}
\displaystyle \lim_{n \to \infty} \frac{m(n)}{2 n}=a \in [0,\, 1].
\end{equation}
They characterized all possible weak*-limit points of the sequence of counting measures associated with distribution of nodes corresponding to a convergent scheme. 

A sequence $\left\{\nu_n\right\}_{n\in\mathbb{N}}$ of measures is said to converge  weakly to the measure $\nu$ provided that there exists a compact set $K$ containing the support of $\nu$ and of each $\nu_n$, and that
\[
\lim_{n\in\mathbb{N}}\int f d\nu_n=\int f d\nu
\]
for each continuous function $f$ on $K$. In such a case, we write $\nu_n\overset{*}{\to}\nu$. We say that $\nu$ is a weak*-limit of the sequence $\left\{\eta_n\right\}_{n\in\mathbb{N}}$ if some subsequence of $\left\{\nu_n\right\}_{n\in \Lambda\subset \mathbb{N}}$ is weakly convergent to $\nu$.

Set two schemes of numbers as in (\ref{nodessequence}) associated to an interpolatory quadrature rule $\left\{I_{n}\right\}_{n\in \Gamma}$ where the degree of exactness satisfies (\ref{limita}) for certain $a\in [0,\, 1]$. We also consider its corresponding sequence $\left\{\eta_n\right\}_{n\in \mathbb{N}}$ of probability counting measures
\begin{equation}\label{zerocountingmeasures}
\eta_n:=\frac{1}{n} \sum_{j=1}^n \delta_{x_{j,n}}, \quad n \in \mathbb{N}.
\end{equation} 
According to \cite{BLS1}, if the rule $\{I_n\}_{n \in \Lambda}$ is convergent then every weak*-limit $\nu$ of the 
the sequence $\left\{\eta_n\right\}_{n\in\mathbb{N}}$ satisfies that 
\begin{equation}\label{sigmainequality}
 \nu \geq  a\, \lambda_0.
\end{equation}
Also from \cite{BLS1} we have that this necessary condition is not sufficient. Theorem \ref{main} states conditions of convergence on the distribution of nodes. 

Let us introduce some previous notation. Set $\mathcal{K}_1$ and $\mathcal{K}_2$ two compact subsets of the complex plane $\mathbb{C}$. Let $\displaystyle \mbox{dist}(\mathcal{K}_1, \mathcal{K}_2)=\min \left\{\left|\left|x-y\right|\right\|: x \in \mathcal{K}_1\,\, \mbox{and}\,\, y\in \mathcal{K}_2\right\}$ denote the distance between $\mathcal{K}_1$ and $\mathcal{K}_2$. Consider a compact set $K \subset \mathbb{C} \setminus [-1,\,1]$, and a measure $\mu$ supported on $K$. A measure $\widetilde{\mu}$ supported on $[-1,\,1]$ is said to be the balayage of $\mu$ if they have the same total variation $||\mu||=||\widetilde{\mu}||$ and their logarithmic  potentials coincide on $[-1,\,1]$. This is
\[
V^{\widetilde{\mu}}(x)=\int \log \frac{1}{|x-t|} \mathrm{d} \widetilde{\mu}(t)=\int \log \frac{1}{|x-\zeta|} \mbox{d} \mu (\zeta)=V^{\mu}(x), \quad x\in [-1,\,1].
\]
In \cite[Section II.4]{ST} we can find a deep study about balayage of measures. We are now ready to state the main result of this paper:

\begin{theorem}\label{main} Fix a number $\kappa \in \mathbb{N}$ and a probability discrete measure 
\[
\sigma=\frac{1}{\kappa}\sum_{k=1}^{\kappa} \delta_{\zeta_k}, \quad \zeta_k \subset \mathbb{C}\setminus [-1,\,1], \quad k=1,\ldots, \kappa. 
\]
Assume that $\sigma$ is symmetric with respect to $\mathbb{R}$ with $\displaystyle \mathrm{dist}\left(\left\{\zeta_1,\ldots,\zeta_{\kappa}\right\},[-1,\,1]\right)>1$. Denote $\widetilde{\sigma}$ the balayage measure associated to $\sigma$ supported on the interval $[-1,\,1]$. Given a rational number $a \in [0,\, 1]$, consider a subsequence $\Lambda \subset \mathbb{N}$ such that for each $n\in \Lambda$, $\displaystyle 2 \frac{1-a}{\kappa} n \in \mathbb{N}$. Let $\displaystyle {\bf x}=\left\{{\bf x}_n=\left(x_{1,n}, \ldots, x_{n,n}\right)\right\}_{n\in \Lambda}$ be a scheme of nodes. If for each $j=1,\ldots n,$ $n\in \Lambda$ there are two constants $A\geq 0$ and $\ell>0$ satisfying  
\begin{equation}\label{maincond}
\left|(1-a)\pi \int^1_{x_{j,n}} \mathrm{d} \widetilde{\sigma}(t) -a \arccos x_{j,n}-\frac{2j-1}{2n}\pi\right|\leq A  e^{-\ell n},
\end{equation}
then there always exist weights $\displaystyle {\bf w}=\left\{{\bf w}_n=(w_{1,n}, \ldots, w_{n,n})\right\}_{n\in \mathbb{N}}$, where $\displaystyle \left\{I_n\right\}_{n\in \Lambda}$ corresponding to ${\bf x}$ and ${\bf w}$ is convergent.
\end{theorem}

In Section \ref{examples} we give some explicit schemes that satisfy the relation (\ref{maincond}). The statement of Theorem \ref{main} is proved in Section \ref{proofmaintheorem}. In such proof we use results coming from the orthogonal polynomials theory that are analyzed in Section \ref{sec:2} and Section \ref{asymptotic analysis}. In Section \ref{sec:2} we study algebraic  properties of families of orthogonal polynomials and their connections with convergent conditions of non-complete interpolatory quadrature rules. In Section \ref{asymptotic analysis} we describe the strong asymptotic behavior of an appropriated family of orthogonal polynomials with respect to a varying measure.

\section{Some explicit convergent schemes of nodes}\label{examples}

We consider three particular cases where the inequality (\ref{maincond}) holds. In the three situations the measure $\sigma=\delta_{\zeta}$ corresponds to a Dirac delta supported on a point belonging to the real line $\zeta>2$. Hence the situations are when $a$ takes the values $0,$ $1/2,$ and $1$. 

According to \cite[Section II.4 equation (4.46)]{ST}, the balayage measure of $\sigma=\delta_{\zeta}$ on $[-1,\,1]$ has the following differential form
\begin{equation}\label{zetadeltabalayage}
\mathrm{d} \widetilde{\sigma}(t)=\frac{\sqrt{\zeta^2-1}}{\pi (\zeta-t)\sqrt{1-t^2}} \, \mathrm{d}\, t.
\end{equation}
We study the function
\[
I_a(x)=(1-a)\pi \int^1_{x} \mathrm{d} \widetilde{\sigma}(t)=(1-a) \sqrt{\zeta^2-1} \int^1_{x} \frac{\mathrm{d}\, t}{(\zeta-t)\sqrt{1-t^2}}.
\]
Taking the change of variables $t=\cos \theta$ and taking into account $\zeta>2$ ($\varphi(\zeta)>2$ implies that $\arg (1-\varphi(\zeta))=\pi$), we have that
\[
I_a(x)=(1-a)\, \left( -\arccos x +2 \arg \left(e^{i \arccos x}-\varphi(\zeta)\right)-2\pi \right).
\]
In this situation the condition of convergence (\ref{maincond}) in Theorem \ref{main} acquires the following form
\begin{equation}\label{mainconddelta}
\left|\arccos x_{j,n}+2(1-a)\, \left[\pi- \arg \left(e^{i \arccos x}-\varphi(\zeta)\right) \right] +\frac{2j-1}{2n}\pi\right|\leq A  e^{-\ell n}.
\end{equation}
Then a scheme $\displaystyle {\bf x}=\left\{{\bf x}_n=\left(x_{1,n}, \ldots, x_{n,n}\right)\right\}_{n\in \Lambda}$ that satisfies the following relation is convergent  
\[
\arccos x_{j,n}+2(1-a)\, \left[\pi- \arg \left(e^{i \arccos x}-\varphi(\zeta)\right) \right]=-\frac{2j-1}{2n}\pi-A e^{-\ell n}=\kappa_{j,n},
\]
with $A>0$ and $\ell>0$. This means that
\[
\cos \left(\arccos x_{j,n}+2(1-a)\, \left[\pi- \arg \left(e^{i \arccos x}-\varphi(\zeta)\right) \right]\right)=\cos \kappa_{j,n}.
\]
Using the cosine addition formula we have that
\[
x_{j,n}\cos \left\{2(1-a)\, \left[\pi- \arg \left(e^{i \arccos x}-\varphi(\zeta)\right) \right]\right\}\]
\begin{equation}\label{generalexample}
-\sqrt{1-x_{j,n}^2}\sin \left\{2(1-a)\, \left[\pi- \arg \left(e^{i \arccos x}-\varphi(\zeta)\right) \right]\right\}=\cos \kappa_{j,n}.
\end{equation}

First we consider the situation $a=1$. In this case the expressions in (\ref{generalexample}) become  
\begin{equation}\label{aonesol}
x_{j,n}=\cos \kappa_{j,n}, \quad j=1,\ldots, n, \quad n\in \Lambda.
\end{equation}
The nodes are close to the zeros of the Chebyshev polynomials. That's why the term corresponding to the $\sigma$'s influence in (\ref{generalexample}) vanishes when $a=1$.

Let us analyze now the case $a=1/2$. We consider the following identities
\begin{equation}\label{condcos}
\cos \left[\pi-\arg \left(e^{i \arccos x_{j,n}}-\varphi(\zeta)\right)\right]=\frac{\varphi(\zeta)-x_{j,n}}{\sqrt{\varphi^2(\zeta)-2 \varphi(\zeta) x_{j,n}+1}}
\end{equation}
and
\begin{equation}\label{condsin}
\sin \left[\pi-\arg \left(e^{i \arccos x_{j,n}}-\varphi(\zeta)\right)\right]=\frac{\sqrt{1-x_{j,n}^2}}{\sqrt{\varphi^2(\zeta)-2 \varphi(\zeta) x_{j,n}+1}}.
\end{equation}
Substituting (\ref{condcos}) and (\ref{condsin}) in (\ref{generalexample}) we arrive at the quadratic equations:
\[
x_{j,n}^2-\frac{2\sin^2\kappa_{j,n}}{\varphi(\zeta)}\, x_{j,n}+ \frac{\sin^2\kappa_{j,n}}{\varphi^{2}(\zeta)} -\cos^2 \kappa_{j,n}=0, \quad j=1,\ldots, n, \quad n\in \Lambda.
\]
For each $j=1,\ldots, n,$ $ n\in \Lambda$ we obtained the following solutions
\begin{equation}\label{ahalfsol}
x_{j,n}=\frac{1}{\varphi(\zeta)}\left[\sin^2 \kappa_{j,n}+ \cos \kappa_{j,n} \sqrt{\varphi^2(\zeta)-\sin^2 \kappa_{j,n}}\right].
\end{equation}
During the process of finding these above solutions we introduce some extra solutions that we removed. Observe that when $\zeta$ tends to $\infty$ the expressions in (\ref{ahalfsol}) reduce to (\ref{aonesol}). This is in accordance with the fact that $\widetilde{\sigma}$ approaches  $\lambda_0$ as $\zeta \to \infty$, see (\ref{zetadeltabalayage}), hence we only considered the positive branch of the square root in (\ref{ahalfsol}).

Finally take $a=0$. From (\ref{generalexample}) we have that
\[
x_{j,n}\cos \left\{2\, \left[\pi- \arg \left(e^{i \arccos x}-\varphi(\zeta)\right) \right]\right\}
\]
\[
-\sqrt{1-x_{j,n}^2}\sin \left\{2\, \left[\pi- \arg \left(e^{i \arccos x}-\varphi(\zeta)\right) \right]\right\}=\cos \kappa_{j,n}.
\]
We use the conditions (\ref{condcos}) and (\ref{condsin}), and obtain the following expression
\[
x_{j,n}=\frac{2 \zeta \cos \kappa_{j,n}}{\varphi(\zeta) + 2 \cos \kappa_{j,n}}, \quad j=1,\ldots, n, \quad n\in \Lambda.
\]
Taking into account that $\displaystyle \varphi(\zeta)=\zeta+\sqrt{\zeta^2+1}$ we see that the above expression is reduced to (\ref{aonesol}) when $\zeta$ goes to infinity.

\section{Connection with orthogonal polynomials}\label{sec:2}

Let $\mu$ be a positive finite Borel measure with infinitely many points in its support $\mbox{supp}(\mu)$. Set $\Delta$ denoting the least interval which contains $\mbox{supp}(\mu)$. A collection of monic polynomials $\displaystyle \left\{q_{\mu,n}\right\}_{n \in \mathbb{Z}_+},$ $\mathbb{Z}_+=\left\{0,1,\ldots\right\}$ is the family orthogonal polynomials with respect to $\mu$ if its elements satisfy the following orthogonality relations
\begin{equation}\label{ortoargentino}
0=\int  x^{\nu} q_{\mu,n}(x) \, \mbox{d} \mu(x), \quad \nu=0,1,\ldots, n-1, \quad n\in \mathbb{Z}_+.
\end{equation}
Each $q_{\mu,n}$ has $n$ single roots lying in the interior of $\Delta$ (we denote $\displaystyle \overset{\circ}{\Delta}$) such that it vanishes at most once in each interval of $\displaystyle \Delta \setminus \mbox{supp}(\mu)$ (see \cite[Theorem 5.2]{chi} or \cite[Chapter 1]{freud}). We also know that $q_{\mu,n+1}$ and $q_{\mu,n}$ interlace their zeros. In \cite{We} B. Wendroff proved that given two polynomials $P_n$ and $P_{n+1}$, with $\deg P_{n+1}=\deg P_n+1=n+1$, that interlace zeros, there always exist measures $\mu$ such that $P_n=q_{\mu,n}$ and $P_{n+1}=q_{\mu,n+1}$. Now we find some of these measures. 

We say then a polynomial $\displaystyle P_n(x)=\prod_{j=1}^n\left(x-x_{j}\right)$ of degree $n$ is admissible with respect to the measure $\mu$, if its roots are all simple, lying in $\displaystyle \overset{\circ}{\Delta}$, with at most one zero into each interval of $\displaystyle \Delta \setminus \mbox{supp}(\mu)$. The system of nodes $(x_{1},\ldots,x_{n})$ is also said to be admissible with respect to $\mu$.

\begin{lemma}\label{interwendroff} Let $\displaystyle P_n(x)=\prod_{j=1}^n \left(x-x_{j}\right)$ and $\displaystyle \widetilde{P}_{n}(x)=\prod_{j=1}^{n-1} \left(x-\widetilde{x}_{j}\right)$ be two admissible polynomials with respect to $\mu$ that satisfy
$\displaystyle x_1<\widetilde{x}_1<x_2< \cdots < \widetilde{x}_{n-1}<x_n.$
Then there exists a positive integrable function $\rho_n$ with respect to $\mu$ ($\rho_n$ is a weight function for $\mu$) such that for the measure $\mu_n$ which differential form $\displaystyle \mbox{d} \mu_n(x)=\rho_n(x) \, \mbox{d} \mu(x),$ $x \in \mbox{supp} (\mu)$, $P_n\equiv q_{\mu_n,n}$ and $\widetilde{P}_{n}\equiv q_{\mu_n, n-1}$ are the $n$-th and $n-1$-th monic orthogonal polynomials with respect to $\mu_n$, respectively.
\end{lemma}

In the proof we follow techniques used in \cite{kroo}.

\begin{proof} Consider $\Phi$ a set of weight functions such that for every constant $\alpha >0$ it satisfies:
\begin{itemize}
    \item[i)] $\rho \in \Phi \implies \alpha \rho \in \Phi.$
    \item[ii)] $(\rho, \widetilde{\rho})\in \Phi^2= \Phi \times \Phi \implies \alpha \rho +(1-\alpha)\widetilde{\rho} \in \Phi,$ $\alpha \leq 1.$
    \item[iii)] If a polynomial $Q$ satisfies $\displaystyle \int Q(x) \rho(x) \, \mbox{d} \mu(x) >0$ for all $\rho \in \Phi,$ then $Q \geq 0$ in $\mbox{supp} (\mu)$.
\end{itemize}

Two examples of sets of weight functions satisfying the above conditions are the positive polynomials and positive simple functions in \cite[Definition 1.16]{rud}. In general, the positive linear combinations of a Chevyshev system (see \cite[Chapter II]{KN}) conform a set as $\Phi$. Examples of Chevyshev systems can be found in \cite{nik} (also in \cite{FMM}).

Given $\rho \in \Phi$ we set
\[
{\bf v}_{\rho}=\left(\int \widetilde{P}_n(x)\rho(x) \mbox{d} \mu(x), \ldots, \int x^{n-2} \widetilde{P}_n(x)\rho(x) \mbox{d} \mu(x),\right.
\]
\[
\left.\int P_{n}(x)\rho(x) \mbox{d} \mu(x), \ldots,\int x^{n-1} P_{n}(x)\rho(x) \mbox{d} \mu(x)\right)\in \mathbb{R}^{2n-1}.
\]
Let us focus on $\mathcal{K}=\left\{{\bf v}_{\rho}:\rho \in \Phi\right\}$. Proving Lemma \ref{interwendroff} reduces to showing that $\mathcal{K}$ contains the origin. From condition $(i)$ we have that the origin belongs to $\mathcal{K}$'s closure, $\overline{\mathcal{K}}$. Since $\mathcal{K}$ is open we need to prove the origin is an interior point. We proceed by contradiction. Suppose that the origin belongs to the boundary of $\mathcal{K}$. This is ${\bf O}\in \partial \mathcal{K}= \overline{K}\setminus \mathcal{K}$. There exists a hyper-plane $\mathcal{A}$ that touches tangentially $\partial \mathcal{K}$ at ${\bf O}$. On the other hand we have that condition (ii) implies that $\mathcal{K}$ is convex, then there exists a vector $\displaystyle {\bf a}=\left(a_{0,n-1}, \ldots, a_{n-2,n-1},a_{0,n},\ldots,a_{n-1,n}\right)$ which is orthogonal with respect to $\mathcal{A}$ in the sense  of the standard inner vector product (${\bf a}\cdot {\bf u}=0$, for all ${\bf u}\in \mathcal{A}$), and for each ${\bf v}_{\rho}\in \mathcal{K}$, ${\bf v}_{\rho} \cdot {\bf a}>0$. So the polynomials
\[
p_{n-1}(x)=a_{0,n-1}+a_{1,n-1}x+ \ldots+a_{n-2,n-1}x^{n-2}
\]
and
\[
p_{n}(x)=a_{0,n}+a_{1,n}x+ \ldots+a_{n-1,n}x^{n-1}
\]
satisfy that
\[
0< \int \left[p_{n-1}(x)\widetilde{P}_{n}(x)+p_{n}P_{n}(x)\right]\, \rho(x) \, \mbox{d} \mu(x), \quad \mbox{for all} \quad \rho \in \Phi.
\]
According to condition (iii) the polynomial $\mathcal{P}(x)=p_{n-1}(x)\widetilde{P}_n(x)+p_{n}P_{n}(x)$, with real coefficients, must be non-negative in $\mbox{supp}(\mu)$. However we shall prove that this is impossible, arriving then to a contradiction.

Assume that $\mathcal{P}(x)=p_{n-1}(x)\widetilde{P}_n(x)+p_{n}P_{n}(x)$ does not change sign in $\mbox{supp}(\mu)$. Suppose that there is a point $t\in \mbox{supp}(\mu)$, such that $t=x_k$ $\displaystyle k\in \left\{1, \ldots, n\right\}$ satisfying that $\mathcal{P}(x_k)=0$, then taking into account that $\widetilde{P}_{n}$ and $P_{n}$ interlace zeros, we have that $p_{n-1}(x_k)=0$. Also, since $\mathcal{P}$ is non-negative on $\mbox{supp}(\mu)$, we have that $t=x_k$ is a zero of multiplicity even for $\mathcal{P}$. Consider $\displaystyle S=\left\{t_1, \ldots, t_{\ell}\right\}\subset \left\{x_1,\ldots,x_{n}\right\}$ the set of all points where  $P_{n}$ and $\mathcal{P}$ vanishe at same time. Then we can write 
\begin{equation}\label{positivepol}
\mathcal{P}=q(x)\prod_{i=1}^{\ell} \left(x-t_{i}\right)^{2 d_i}, \quad d_i\in \mathbb{N}, \quad i=1,\ldots,\ell,
\end{equation}
where $q$ is a polynomial with positive values at every root of $P_{n}$. We also write
\begin{equation}\label{pnexpand}
p_{n-1}(x)=\widetilde{p}(x)\prod_{i=1}^{\ell} \left(x-t_{i}\right).
\end{equation}
The polynomial $\widetilde{p}$ has degree $\deg p_n-\ell$. Since $\widetilde{P}_{n}$ and $P_n$ interlace zeros, we have that
\begin{equation}\label{quotient}
\frac{\widetilde{P}_n(x)}{P_{n}(x)}=\sum_{j=1}^{n} \frac{\lambda_{j}}{x-x_{j}}, \quad \lambda_{j}>0, \quad j=1,\ldots,n.
\end{equation}
Observe that
\[
\frac{\displaystyle \prod_{i=1}^{\ell}\left(x-t_i\right)}{P_{n}(x)}=\frac{1}{\displaystyle  \prod_{i=1}^{\ell}\left(x-t_i\right)^{2d_i-1}q(x)}\left[p_{n-1}(x)\sum_{j=1}^{n} \frac{\lambda_{j}}{x-x_{j}}+p_{n}(x)\right].
\]
This means that the above function satisfies that
\[
\frac{1}{\displaystyle  \prod_{i=1}^{\ell}\left(z-t_i\right)^{2d_i-1}q(z)}\left[p_{n-1}(z)\sum_{j=1}^{n} \frac{\lambda_{j}}{z-x_{j}}+p_{n}(z)\right]=\mathcal{O}\left(\frac{1}{z^{n-\ell}}\right) \quad \mbox{as} \quad z \to \infty,
\]
which is a holomorphic functions on $\displaystyle \overline{\mathbb{C}}\setminus \left(\left\{x_1,\ldots,x_{n}\right\}\setminus S\right).$ For each $\nu=0,\ldots, n-\ell-2$ we have then
\[
\frac{z^{\nu}}{\displaystyle  \prod_{i=1}^{\ell}\left(z-t_i\right)^{2d_i-1}q(z)}\left[p_{n-1}(z)\sum_{j=1}^{n} \frac{\lambda_{j}}{z-x_{j}}+p_{n}(z)\right]=\mathcal{O}\left(\frac{1}{z^{2}}\right) \quad \mbox{as} \quad z\to \infty,
\]
also holomorphic functions on $\displaystyle \overline{\mathbb{C}}\setminus \left(\left\{x_1,\ldots,x_{n}\right\}\setminus S\right).$ Set the elements $\displaystyle y_j \in \left\{x_1,\ldots,x_{n}\right\}\setminus S$, $j=1, \ldots, n-\ell$ with $y_1 < y_2 < \cdots < y_{n-\ell}$, and $\widetilde{\lambda}_j$, $j=1, \ldots, n-\ell$  the coefficients $\lambda$'s defined in (\ref{quotient}) corresponding to points $y_j$. Also let $\lambda_j^{\prime}$ denote the $\lambda$'s of $t_j$, $j=1, \ldots, \ell$. Call $F$ the set of the roots of the polynomial $q$ defined in (\ref{positivepol}). Consider a closed integration path $\Gamma$ with winding number $1$ for all its interior points. Denote $\mbox{Ext} (\Gamma)$ and $\mbox{Int}(\Gamma)$ the unbounded and bounded connected components respectively of the complement of $\Gamma$. Take $\Gamma$ so that $I\subset \mbox{Int}(\Gamma)$ and   $F\subset \mbox{Ext}(\Gamma)$. From Cauchy's Theorem and the above two conditions, it follows that
\[
0=\frac{1}{2 \pi i} \int_{\Gamma} \frac{z^{\nu}}{\displaystyle  \prod_{i=1}^{\ell}\left(z-t_i\right)^{2d_i-1}q(z)}\left[p_{n-1}(z)\sum_{j=1}^{n} \frac{\lambda_{j}}{z-x_{j}}+p_{n}(z)\right] \, \mbox{d} z
\]
\[
=\frac{1}{2 \pi i} \int_{\Gamma} \frac{\displaystyle z^{\nu} p_{n-1}(z)\sum_{j=1}^{n-\ell} \frac{\widetilde{\lambda}_{j}}{z-y_{j}}\,  \mbox{d} z}{\displaystyle  \prod_{i=1}^{\ell}\left(z-t_i\right)^{2d_i-1}q(z)} + \frac{1}{2 \pi i} \int_{\Gamma} \frac{z^{\nu} p_{n}(z) \, \mbox{d} z}{\displaystyle  \prod_{i=1}^{\ell}\left(z-t_i\right)^{2d_i-1}q(z)} .
\]
Since $\displaystyle \frac{z^{\nu} p_{n}(z)}{\displaystyle  \prod_{i=1}^{\ell}\left(z-t_i\right)^{2d_i-1}q(z)}  \in \mathcal{H} \left(\mbox{Int}(\Gamma)\right)$ the second term vanishes. From (\ref{pnexpand}), using the Cauchy integral formula, we obtain: 
\[
0=\sum_{j=1}^{n-\ell} y_j^{\nu}\widetilde{p}(y_j)\frac{\widetilde{\lambda}_{j}}{\displaystyle  \prod_{i=1}^{\ell}\left(y^j-t_i\right)^{2(d_i-1)}q(y_j)}=0, \quad \nu=0,\ldots, n-\ell-1.
\]
Taking into account that for each $j=1, \ldots, n-\ell$, $\displaystyle \frac{\widetilde{\lambda}_{j}}{\displaystyle  \prod_{i=1}^{\ell}\left(y^j-t_i\right)^{2(d_i-1)}q(y_j)}>0$, we conclude that the above orthogonality relations imply that $\widetilde{p}$ must change sign at least $n-\ell$ times, hence $\deg \widetilde{p} \geq n-\ell$. Since $\deg \widetilde{p}= \deg p_{n-1}-\ell\leq n-\ell-1$ we arrive at a contradiction which completes the proof.
\end{proof}

Consider a monic polynomial $\displaystyle P_n(x)=\prod_{j=1}^n(x-x_{j})$ with degree $n\in \mathbb{N}$ which is $\mu$ admissible. We say that a weight function $\rho_n$ on $\mbox{supp}(\mu)$ is orthogonal with respect to $P_n(x)$ and $\mu$ if $P_n \equiv q_{\mu_n,n}$, where $\mbox{d} \mu_n(x)=\rho_n(x) \, \mbox{d} \mu(x)$, $x\in \mbox{supp}(\mu)$. We also say that $\rho_n$ is orthogonal with respect to ${\bf x}_n=(x_{1}, \ldots, x_{n})$ and $\mu$. A sequence of weight functions $\left\{\rho_n\right\}_{n\in \mathbb{N}}$ is a family of orthogonal weight functions with respect to the sequence of polynomials $\left\{P_n\right\}_{n\in \mathbb{N}}$, if for each $n\in \mathbb{N}$, $P_n\equiv q_{\mu_n,n}$.   

Let $q_{m(n)}$ be an arbitrary polynomial with degree $\deg q_{m(n)}(x)=2n-m(n)-1$ being positive on $[-1,\,1]$. Let $\mu_n$ denote the measure with differential form $\mbox{d} \mu_n(x) = q_{m(n)}^{-1}(x) \mbox{d} \mu(x)$, $x\in \mbox{supp}(\mu)$. Set a system of nodes ${\bf x}_n=(x_{1}, \ldots, x_{n})$ such that $\displaystyle P_n(x)=\prod_{j=1}^n(x-x_{j})=q_{\mu_n,n}$. This means that ${\bf x}_n$ is the system of $n$ nodes corresponding to the Gaussian quadrature rule for the measure $\mu_n$. Given an arbitrary polynomial $p\in \Pi_{m(n)}$, we have that
\begin{equation}\label{supergauss}
q_{m(n)}(x)p(x)-\sum_{j=1}^n q_{m(n)}(x_{j})p(x_{j})L_{j,n}(x)=q_{\mu_n,n}(x)\mathcal{P}_{n-1}(x),
\end{equation}
where $\displaystyle L_{j,n}(x):=\prod_{\underset{k\not=j}{k=1}}^n \frac{x-x_{k}}{x_{j}-x_{k}},\quad j=1,\ldots,n,$ and $ \mathcal{P}_{n-1}$ is a certain polynomial with $ \displaystyle \deg \mathcal{P}_{n-1}=n- (m(n)-\deg p)-1\leq n-1.$

Observe that $\displaystyle \int p(x) \, \mbox{d}\,  \mu (x)= \int q_{m(n)}(x)p(x) \mbox{d} \mu_n(x).$ Hence from (\ref{supergauss}) we obtain
\[
\int p(x) \, \mbox{d}\,  \mu (x)- \int \sum_{j=1}^n q_{m(n)}(x_{j})p(x_{j})L_{j,n}(x)\,  \mbox{d}\,  \mu_n (x)=\int q_{\mu_n,n}(x)\mathcal{P}_{n-1}(x)\, \mbox{d}\, \mu_n(x),
\]
which vanishes because $q_{\mu_n,n}$ satisfies the orthogonality relations for $\mu_n$ as in (\ref{ortoargentino}).  We conclude then
\[
\int p(x)\, \mbox{d}\, \mu(x) =\sum_{j=1}^n p(x_{j,n})q_{m(n)}(x_{j,n}) \int L_{j,n}(x)\frac{\mbox{d}\,  \lambda_0 (x)}{q_{m(n)}(x)}=\sum_{j=1}^n w_{j,n} p(x_{j,n}).
\]
This is an interpolatory integration rule with degree of exactness $m(n)$, where the weights can be defined via
\begin{equation}\label{lambda}
w_{j,n}=q_{m(n)}(x_{j})\int L_{j,n}(x)\, \mbox{d} \mu_n(x)=q_{m(n)}(x_{j}) \widetilde{w}_{j,n}, \quad j=1,\ldots,n.
\end{equation}
The numbers $\widetilde{w}_{j,n}$, $j=1, \ldots,n$ are the weights corresponding to a Gaussian quadrature rule, which are all positive. Since $q_{m(n)}$ is also positive the weights $w_{j,n}>0$. According to P\'olya's condition a sequence of these rules of integration is convergent. 

Let us consider $\displaystyle {\bf x}=\left\{{\bf x}_n=\left(x_{1,n}, \ldots, x_{n,n}\right)\right\}_{n \in \mathbb{N}}$ an admissible scheme of nodes for a measure $\mu$, and take a corresponding family of orthogonal weights $\left\{\rho_n\right\}_{n \in \mathbb{N}}$. For each $n$, $\mu_n$ denotes the measure with differential form $\mbox{d} \mu(x)= \rho_n(x) \mbox{d} \mu(x)$, and introduce its family of orthonormal polynomials $\displaystyle \left\{p_{\mu_n,j}\right\}_{j\in \mathbb{Z}_+}$. This means that $\displaystyle p_{\mu_n,j}\equiv q_{\mu_n,j}/ \left|\left|q_{\mu_n,j}\right|\right|_{2, \mu_n},$ $j\in \mathbb{Z}_+$  where  $||\cdot||_{2,\mu_n}$ denotes the $\mbox{L}_2$ norm corresponding to the measure $\mu_n$. 

Given a function $f \in \mbox{L}_{2,\mu_n}$ and $j\in \mathbb{Z}_+$ we consider the $j$-th partial sum of the Fourier series corresponding to $f/\rho_n$ on the bases $\displaystyle \left\{p_{\mu_n,j}\right\}_{j\in \mathbb{Z}_+}$:
\[
S_{f, \mu_n, j}=\sum_{k=0}^{j-1} f_k p_{\mu_n,k}(x), \quad f_k=\int f(x) p_{\mu_n, k}(x) \mbox{d} \mu_n(x), \quad k=0,\ldots, j-1.
\]
Using the Christoffel-Darboux identity (see \cite[Theorem 3.2.2]{Sze}) we can deduce
\begin{equation}\label{CDidentity}
S_{f,\mu_n,j}(x)=\int \frac{q_{\mu_n,j}(x)q_{\mu_n,j-1}(t)-q_{\mu_n,j}(t)q_{\mu_n,j-1}(x)}{\left|\left|q_{\mu_n,j-1}\right|\right|_{2,\mu_n}^2(x-t)} f(t) \mbox{d} \mu_n(t).  
\end{equation}
The following result is an extension of \cite[Theorem 15.2.4 (equality 15.2.7)]{Sze}

\begin{lemma}\label{mainsec} Let $\displaystyle (x_{1},\ldots,x_{n})$ be an $\mu$ admissible system of nodes. Given a polynomial $\displaystyle q_{m(n)}$ take the system of weights $\displaystyle (w_{1,n},\ldots,w_{n,n})$ whose elements $w_{j,n}$, $j=1, \ldots,n$, are constructed using (\ref{lambda}). Then there always exists a weight $\rho_n$ such that
\begin{equation}\label{identityknown}
\displaystyle \frac{w_{j,n}}{q_{m(n)}(x_{j})}=\frac{\left|\left|q_{\tau_n,n-1}\right|\right|_{2,\tau_n}^2S_{1/\rho_n, \tau_n,n}(x_{j})}{q_{\tau_n,n-1}(x_{j})q_{\tau_n,n}^{\prime}(x_{j})}=-\frac{\left|\left|q_{\tau_n,n}\right|\right|_{2,\tau_n}^2S_{1/\rho_n, \tau_n,n+1}(x_{j})}{q_{\tau_n,n+1}(x_{j})q_{\tau_n,n}^{\prime}(x_{j})},
\end{equation}
where the measure $\tau_n$ is such that  $\displaystyle \frac{\displaystyle \mbox{d} \tau_n}{\mbox{d} \mu}= \frac{\rho_n}{q_{m(n)}}$. Thus $\mbox{sign}\, w_{j,n}= \mbox{sign}\, S_{1/\rho_n, \tau_n,n}(x_{j})=\mbox{sign}\, S_{1/\rho_n, \tau_n,n+1}(x_{j}),$ $j=1,\ldots, n$.
\end{lemma}

\begin{proof} Take an orthogonal weight $\rho_n$ with respect to the system of $n$ nodes $(x_{1},\ldots,x_{n})$ and the measure with differential form $\mbox{d} \mu(x)/q_{m(n)}(x)$. According to (\ref{lambda}) and taking into account that $P_n \equiv q_{\mu_n,n}$ where the measure $\tau_n$ has the differential form $\displaystyle \mbox{d} \tau_n(x) =\rho_n(x) \mbox{d} \mu (x)/q_{m(n)}(x)$, we have the following
\[
w_{j,n}=q_{m(n)}(x_{j})\int \frac{q_{\tau_n,n}(x)}{q_{\tau_n,n}^{\prime}(x_{j})(x-x_{j})} \, \frac{d\mu(x)}{q_{m(n)}(x)}, \quad j=1,\ldots,n.
\]
Arranging the above formula and using the identity (\ref{CDidentity}) we obtain that
\[
w_{j,n}=\frac{q_{m(n)}(x_{j})\left|\left|q_{\tau_n,n}\right|\right|_{2,\tau_n}^2}{q_{\tau_n,n+1}(x_{j})q_{\tau_n,n}^{\prime}(x_{j})} \int \frac{q_{\tau_n,n+1}(x_{j})q_{\tau_n,n}(x)}{\left|\left|q_{\tau_n,n}\right|\right|_{2,\tau_n}^2(x-x_{j})} \frac{1}{\rho_n(x)} \frac{\rho_n(x)\, d\mu(x)}{q_{m(n)}(x)}
\]
\[
=-\frac{q_{m(n)}(x_{j})\left|\left|q_{\tau_n,n}\right|\right|_{2,\tau_n}^2}{q_{\tau_n,n+1}(x_{j})q_{\tau_n,n}^{\prime}(x_{j})} S_{1/\rho_n, \tau_n,n}(x_{j}),
\]
which proves the second identity in (\ref{identityknown}). Since $q_{\tau_n,n+1}(x_{j})q_{\tau_n,n}^{\prime}(x_{j})<0, \quad j=1,\ldots,n,$ then $\mbox{sign}\,  w_{j,n}= \mbox{sign}\, S_{1/\rho_n, \tau_n,n+1}(x_{j})$. Following the above steps we can prove the first equality in (\ref{identityknown}) and $\mbox{sign}\,  w_{j,n}= \mbox{sign}\, S_{1/\rho_n, \tau_n,n}(x_{j})$.
\end{proof}

The following two results are consequences of the above Lemma \ref{mainsec}

\begin{lemma}\label{theokey1} An admissible scheme of nodes $\displaystyle {\bf x}=\left\{{\bf x}_n=(x_{1,n},\ldots,x_{n,n})\right\}_{n\in \mathbb{N}}$ is convergent if there exists a family of orthogonal weights $\displaystyle \left\{\rho_n\right\}_{n\in \mathbb{N}}$ with respect to ${\bf x}$ and the sequence of measures $\displaystyle \left\{\mbox{d} \tau_n(x) =\mbox{d} \mu (x) / q_{m(n)}(x)\right\}_{n\in \mathbb{N}}$ satisfying 
\begin{equation}\label{keyeq}
    \lim_{n \to \infty} \left|\left|1-\rho_n(x)S_{1/\rho_n, \tau_n,n}(x)  \right|\right|_{[-1,\,1],\infty}=0,
\end{equation}
where $\displaystyle \left|\left|\cdot\right|\right|_{[-1,1],\infty}$ denotes the supremum norm on $[-1,\,1]$.
\end{lemma}

\begin{proof} Assuming the equality (\ref{keyeq}), there exists a number $N>0$ such that for every $n \geq N$ the function $S_{1/\rho_n, \mu_n,n}(x)>0$ on $[-1,\,1]$ particularly at the nodes. According to Lemma \ref{mainsec}, the coefficients $w_{j,n}$, $j=1, \ldots, n$, are also positive. This completes the proof.
\end{proof}

\begin{lemma}\label{lemaclave} Consider the varying measure $\mbox{d}\, \mu_n(x)=\mbox{d} \, \mu(x)/q_{m(n)}(x)$ and their orthogonal polynomials $\displaystyle q_{\mu_n,n}(x)=\prod_{j=1}^n\left(x-x_{j,n}\right)$ and $\displaystyle q_{\mu_n,n-1}(x)=\prod_{j=1}^{n-1}\left(x-x_{j,n-1}\right)$, $n\in \mathbb{N}$. Let $\displaystyle {\bf y}=\left\{{\bf y}_n=\left(y_{1,n},\ldots, y_{n,n}\right)\right\}_{n\in \mathbb{N}}$ be a scheme of nodes such that for each $n\in \mathbb{N}$
\begin{equation}\label{entrelazamiento}
-1<y_{1,n}<x_{1,n-1}<y_{2,n}<\cdots < x_{n-1,n-1}<y_{n,n} <1.
\end{equation}
Assume that the polynomials $\displaystyle P_{n}(x)=\prod_{j=1}^n \left(x-y_{j,n}\right), \quad n\in \mathbb{N}$
satisfy
\begin{equation}\label{extructura}
\lim_{n \to \infty}\frac{1}{\left|\left|q_{\mu_n, n-1}\right|\right|^2_{2, \mu_n}}\left[\frac{(q_{\mu_n,n}-P_n)q_{\mu_n,n-1}}{q_{m(n)}^2}\right]^{\prime}=0, \quad \mbox{on} \quad [-1,\,1].
\end{equation}
Then ${\bf y}$ is convergent.
\end{lemma}

\begin{proof} From Lemma \ref{interwendroff} we ensure the existence of a weight function $\rho_n$ such that the polynomials $q_{\tau_n,n-1}$ and $P_{n}$ belong to the family of orthogonal polynomials corresponding to the measure $\rho_n(x)\mbox{d} \mu (x)/q_{m(n)}(x)$. Let us analyze the function
\[
\left|1-\frac{\left|\left| P_{n}\right|\right|^2_{2, \rho_n d \mu/q_{m(n)}}}{q_{m(n)}(x) \left|\left|q_{\mu_n, n-1}\right|\right|^2_{2, \mu_n}}S_{1/\rho_n, \rho_n d \mu/q_{m(n),n} }(x)\right|
\]
\[
=\frac{1}{q_{m(n)}(x)}\left|S_{q_{m(n)}, \mu_n,n}(x) -\frac{\left|\left| \widetilde{P}_{n}\right|\right|^2_{2, \rho_n d \mu/q_{m(n)}}}{\left|\left|q_{\mu_n, n-1}\right|\right|^2_{2, \mu_n}}S_{1/\rho_n, \rho_n d \mu/q_{m(n),n} }(x)\right|
\]
We have used that $\displaystyle S_{q_{m(n)}, \mu_n,n}\equiv q_{m(n)}$, hence  we need to show that  
\[
\lim_{n\to 0}\frac{1}{q_{m(n)}(x)}\left|S_{1/\rho_n, \mu_n,n}(x)-\frac{\left|\left|\widetilde{P}_{n}\right|\right|^2_{2, \rho_n d \mu/q_{m(n)}}}{\left|\left|q_{\mu_n, n-1}\right|\right|^2_{2, \mu_n}}S_{1/\rho_n, \rho_n d \mu/q_{m(n),n} }(x)\right|=0.
\]
Applying (\ref{CDidentity}) we observe that
\[
\frac{1}{q_{m(n)}(x)}\left(S_{1/\rho_n, \mu_n,n}(x)-\frac{\left|\left|\widetilde{P}_{n}\right|\right|^2_{2, \rho_n d \mu/q_{m(n)}}}{\left|\left|q_{\mu_n, n-1}\right|\right|^2_{2, \mu_n}}S_{1/\rho_n, \rho_n d \mu/q_{m(n),n} }(x)\right)=\frac{q_{m(n)}^{-1}(x)}{\left|\left|q_{\mu_n,n-1}\right|\right|_{2,\mu_n}^2}
\]
\[
\times \int \left( \frac{q_{\mu_n,n}(x)q_{\mu_n,n-1}(t)-q_{\mu_n,n}(t)q_{\mu_n,n-1}(x)}{x-t}-\frac{P_{n}(x)q_{\mu_n,n-1}(t)-P_{n}(t)q_{\mu_n,n-1}(x)}{x-t}\right)
\]
\[
\times \frac{\mbox{d} \mu(t)}{q_{m(n)}(t)}
\]
Let us consider the kernel
\[
\frac{q_{\mu_n,n}(x)q_{\mu_n,n-1}(t)-q_{\mu_n,n}(t)q_{\mu_n,n-1}(x)}{\left|\left|q_{\mu_n,n-1}\right|\right|_{2,\mu_n}^2q_{m(n)}(x)q_{m(n)}(t)(x-t)}-\frac{P_{n}(x)q_{\mu_n,n-1}(t)-P_{n}(t)q_{\mu_n,n-1}(x)}{\left|\left|q_{\mu_n,n-1}\right|\right|_{2,\mu_n}^2q_{m(n)}(x)q_{m(n)}(t)(x-t)}
\]
\[
=\frac{(q_{\mu_n,n}-P_n)(x)q_{\mu_n,n-1}(t)-(q_{\mu_n,n}-P_n)(t)q_{\mu_n,n-1}(x)}{\left|\left|q_{\mu_n,n-1}\right|\right|_{2,\mu_n}^2q_{m(n)}(x)q_{m(n)}(t)(x-t)}=\mathcal{K}(x,t).
\]
From Taylor's Theorem we obtain that
\[
\mathcal{K}(x,t)= \frac{1}{\left|\left|q_{\mu_n, n-1}\right|\right|^2_{2, \mu_n}}\left[\frac{(q_{\mu_n,n}-P_n)q_{\mu_n,n-1}}{q_{m(n)}^2}\right]^{\prime}(s)
\]
for some $s$ in between of $x$ and $t$, so the assumption (\ref{extructura}) completes the proof. 
\end{proof}

\section{Asymptotic analysis}\label{asymptotic analysis}

Let us consider the varying measure $\mu_n$ with $\displaystyle \mbox{d} \mu_n(x)/\mbox{d} x= (q_{m(n)}(x)\sqrt{1-x^2})^{-1}$, where $\displaystyle q_{m(n)}(x)= q_{\kappa}^{\displaystyle 2 \frac{1-a}{k} n}(x)=\left(\prod_{k=1}^{\kappa}(x-\zeta_j)\right)^{\displaystyle 2 \frac{1-a}{k} n}, \quad n\in \Lambda.$
Let $\sigma$ be the zero counting measure of $q_k$. This is $\displaystyle \sigma=\frac{1}{\kappa}\sum_{k=1}^{\kappa} \delta_{\zeta_k}.$ Set the analytic logarithmic potential  corresponding to the measure $\sigma$:
\begin{equation}\label{externalfield}
g(z,{\sigma})= -\int \log\, (z-\zeta)\, \mbox{d}\, \sigma(\zeta).
\end{equation}
We take the logarithmic branch such that $g(z,{\sigma})$ is analytic on a domain $D \subset K$ that contains the interval $[-1,\,1]$, and also for every $x \in [-1,\,1]$,
\begin{equation}\label{potentialsigma}
V^{\sigma}(x)= \int \log\, \frac{1}{|x-\zeta|} \mbox{d}\, \sigma(\zeta)=g(x,{\sigma})= -\int \log\, (z-\zeta)\, \mbox{d}\, \sigma(\zeta).
\end{equation}
Since $\sigma$ is symmetric we $\displaystyle \int \arg (x-\zeta) \mbox{d}\, \sigma(\zeta)=0$. In each compact $K\subset D$ we have that $\displaystyle \frac{1}{2n} \log  \frac{1}{q_{m(n)}(z)}=(1-a) \,g(z,{\sigma})$  on  $K.$

\begin{lemma}\label{explicitexp} Let $\displaystyle \mbox{d} \mu_n(x)/\mbox{d} x= (q_{m(n)}(x)\sqrt{1-x^2})^{-1}$, $n\in \mathbb{N}$ be a sequence of measures as above. Then
\begin{equation}\label{qmunn}
q_{\mu_n,n}= \left(1+\mathcal{O}(e^{-c n})\right) \exp\left\{ -n V^{\overline{\nu}}\right\}K_{1,n}+\mathcal{O}(e^{-c n})\exp\left\{- n V^{\overline{\nu}}\right\} K_{2,n}
\end{equation}
and
\begin{equation}\label{qmunn-1}
 \frac{d_{n,n-1}}{2^{2 n a}} q_{\mu_n,n-1}=\left(1+\mathcal{O}(e^{-c n})\right) \exp\left\{ -n V^{\overline{\nu}}\right\}K_{2,n}+\mathcal{O}(e^{-c n})\exp\left\{- n V^{\overline{\nu}}\right\} K_{1,n}
\end{equation}
where $\displaystyle d_{n,n-1}=-\left(2 \pi i \left|\left|q_{\mu_n,n-1}\right|\right|^2_{\mu_n,2}\right)^{-1},$
\begin{equation}\label{forK1}
K_{1,n}(x)=\displaystyle 2  \cos n\left((1-a)\pi \int^1_x \mbox{d} \widetilde{\sigma}(t) -a \arccos{x} \right),
\end{equation}
and
\begin{equation}\label{forK2}
K_{2,n}(x)=\displaystyle  \frac{1}{i} \cos n \left((1-a) \pi \int_{x}^1 \mbox{d} \widetilde{\sigma}(t) - (a-1/n) \arccos{x} \right).
\end{equation}
\end{lemma}

\begin{proof}
We study a matrix Riemann-Hilbert problem like in \cite[Theorem 2.4]{KMVV} whose solution $Y$ is a $2 \times 2$ matrix function satisfying the following conditions: 
\begin{enumerate}
\item $Y \in \mathcal{H}(\mathbb{C}\setminus[-1,\,1])$ (all the entries of $Y$ are analytic on $\mathbb{C}\setminus[-1,\,1]$),
\item $\displaystyle Y_+(x)=Y_-(x)\left(\begin{array}{c c}
1 & \left(q_{m(n)}(x)\sqrt{1-x^2}\right)^{-1}\\ & \\
0 & 1
\end{array}\right)$, $x \in (-1,\,1)$,
\item $\displaystyle Y(z)\left(\begin{array}{c c}
z^{-n} & 0\\ & \\
0 & z^{n}
\end{array}\right)=\mathbb{I}+\mathcal{O}(1/z)$ as $z \to \infty$, $\mathbb{I}$ is the $2\times2$ identity matrix.
\item $\displaystyle Y(z)=\mathcal{O}\left(\begin{array}{c c}
1 & |z\pm 1|^{-1/2} \\ & \\
1 &  |z\pm 1|^{-1/2}
\end{array}\right)$ as $z \to \mp 1$.
\end{enumerate}

According to \cite[Theorem 2.4]{KMVV} (see also \cite{Ku}) the $Y$ solution of above matrix Riemann-Hilbert problem (for short Y-RHP) is unique and has the form
\[
Y(z)=\left(\begin{array}{c c}
q_{\mu_n,n}(z) & \displaystyle -\frac{1}{2 \pi i}\int \frac{q_{\mu_n,n}(x)}{z-x} \, \mbox{d}\mu_n(x) \\ & \\
d_{n,n-1} q_{\mu_n,n-1}(z) & \displaystyle -\frac{d_{n,n-1}}{2 \pi i}\int \frac{q_{\mu_n,n-1}(x)}{z-x} \, \mbox{d}\mu_n(x) 
\end{array}\right).
\] 

The key of our procedure follows the ideas introduced in \cite{AV}. We find a relationship between $Y$ and the matrix solution $R: \mathbb{C}\setminus \gamma \rightarrow \mathbb{C}^{2\times 2}$ corresponding to another Riemann-Hilbert problem (R-RHP) for a closed Jordan curve $\gamma$ positively oriented surrounding the interval $[-1,\,1]$: 
\begin{enumerate}
\item $R \in \mathcal{H} (\mathbb{C}\setminus \gamma)$,
\item $R_+(\zeta)=R_-(\zeta)V_n(\zeta),$ $\zeta \in \gamma$, with $V_n\in \mathcal{H}(D)$,
\item $R(z) \rightarrow \mathbb{I}$ as $z \to \infty$,
\end{enumerate}
where $V_n=\mathbb{I}+\mathcal{O}(\varepsilon^n)$ with $0\leq \varepsilon<1$, uniformly on compact subsets of $K$ as $n \to \infty$. Those conditions imply that $R=\mathbb{I}+\mathcal{O}(\varepsilon^n)$  uniformly on $\mathbb{C}$ as $n \to \infty$. There is a chain of transformations to arrive from $Y$ to $R$, which we represent $Y \rightarrow T \rightarrow S \rightarrow R$. Once we have arrived to $R$, we recover the entries of $Y$ going back from $R$ to $Y$. 

From \cite[Corollary 4]{BLS2} we have that the zero counting measures $\nu_n$ defined in (\ref{zerocountingmeasures}) corresponding to the monic orthogonal polynomials $\displaystyle q_{\mu_n,n}(z)=\prod_{j=1}^n\left(z-x_{j,n}\right)$ with respect to the varying measures $\mu_n$, satisfy 
\begin{equation}\label{equilibriummeasure}
\nu_n \overset{\star}{\rightarrow} \overline{\nu}=(1-a) \widetilde{\sigma} +a \lambda_0 \quad \mbox{as} \quad n \to \infty,
\end{equation}
where $\widetilde{\sigma}$ denotes the balayage of the measure $\sigma$ out of $\mathbb{C}\setminus [-1, \,1]$ onto $[-1,\,1]$. 

The measure $\overline{\nu}$ is the so called (see \cite[Theorem I.1.3]{ST}) equilibrium measure under the influence of the external field $ (1-a) V^{\sigma}(z)$. From (\ref{equilibriummeasure}) we have the following equilibrium condition
\begin{equation}\label{equilibriumcondition}
V^{\overline{\nu}}(t)-(1-a) V^{\sigma}(t)=a V^{\lambda_0}(t)=a \log 2, \quad t\in [-1,\,1].
\end{equation}

Observe the conditions $(3)$ in both Riemann Hilbert problems. $Y$ requires a normalization at infinity to get to $R$'s behavior at infinity. We modify $Y$ to obtain a Riemann-Hilbert problem whose solution is defined on the same set as $Y$, which approaches $\mathbb{I}$ as $n \rightarrow \infty$.  Let us introduce the function $g(z,\overline{\nu})$, which is the analytic potential corresponding to the measure $\overline{\nu}$ described in (\ref{equilibriummeasure})
\begin{equation}\label{defgn}
g(z,\overline{\nu})=-\int \log\, (z-t)\, \mbox{d} \overline{\nu}(t)=V^{\overline{\nu}}(z)-\,i\,\int \arg (z-t) \, \mbox{d}\overline{\nu}(t),
\end{equation}
with $\arg$ denoting the principal argument $g(z,\overline{\nu}) \in \mathcal{H}(K\setminus (-\infty,\,1])$. Substituting $g(z,\overline{\nu})$ in (\ref{equilibriumcondition}) we obtain 
\begin{equation}\label{roi}
g_+(x,\overline{\nu})+g_-(x,\overline{\nu})-2a \log 2-2(1-a)g(x,\sigma)=0, \quad x \in [-1,\,1].
\end{equation}
and
\begin{equation}\label{casta}
g_-(x,\overline{\nu})-g_+(x,\overline{\nu})=\left\{\begin{array}{l l l}
0 & \mbox{if} & x\geq 1 \\
2\pi i &\mbox{if} & x\leq 1 \\
2 \wp(x) &\mbox{if} & x\in (-1,\,1),
\end{array}
\right.
\end{equation}
with
\begin{equation}\label{california}
\wp(x)=\displaystyle \pi i \int^1_{x} \mbox{d} \overline{\nu}(t)=\pi i \left[(1-a) \displaystyle  \int^1_{x} \mbox{d} \widetilde{\sigma}(t)-\frac{a}{\pi}\arccos(x)  \right].
\end{equation}

Consider the matrices $\displaystyle G(z)=\left(\begin{array}{c c}
e^{n g(z,\overline{\nu})} & 0 \\ 
0 & e^{-n g(z,\overline{\nu})} 
\end{array}
\right)$  and $\displaystyle L=\left(\begin{array}{c c}
2^{n a} & 0 \\ 
0 & 2^{-n a} 
\end{array}
\right).$
We define the matrix function $T=LYGL^{-1}$. So $T$ is the unique solution of the following Riemann-Hilbert problem (T-RHP)
\begin{enumerate}
\item $T \in \mathcal{H}(\mathbb{C}\setminus [-1,\,1])$,
\item $T_+(x)=T_-(x)M(x)$, $x\in (-1,\,1),$
\item $T(z)=\mathbb{I}+\mathcal{O}\left(1/z\right)$ as $z \to \infty$,
\item $\displaystyle T(z)=\mathcal{O}\left(\begin{array}{c c}
1 &|z\pm 1|^{-1/2} \\
1 & |z\pm 1|^{-1/2} 
\end{array}\right)$ as $z \to \mp 1$,
\end{enumerate}
where according to (\ref{roi}) and (\ref{casta}) the jump matrix $\displaystyle M(x)=\left(\begin{array}{c c}
e^{-2 n \wp(x)} & (1-x^2)^{-1/2}  \\
0 & e^{2 n\wp(x)} 
\end{array}
\right),$  $x\in (-1,\,1).$ 

According to \cite[Theorem 1.34]{DKMc} there exists a domain $D$ containing the interval $[-1,\,1]$ where the function $\wp$ in (\ref{california}) admits an analytic extension on $D\setminus [-\infty,\,1]$ as
\begin{equation}\label{defwp}
\mathcal{A}(z)=\pi i \int_z^1 \, \mbox{d}\overline{\nu}(\zeta)=\pi i \int_z^1 \, \overline{\nu}^{\prime}(\zeta) \mbox{d} \zeta,
\end{equation}
where $\displaystyle \overline{\nu}^{\prime}(\zeta)=\frac{\psi(\zeta)}{\sqrt{1-\zeta^2}},$ with $\psi \in \mathcal{H}(D)$ and $\psi(x)>0,$ $x\in (-1,1).$ Observe that $\mathcal{A}_+(x)=\wp(x)=-\mathcal{A}_-(x)$, then  we write $\displaystyle M(x)=\left(\begin{array}{c c}
e^{-2 n \mathcal{A}_+(x)} & (1-x^2)^{-1/2}\\
0 & e^{-2 n\mathcal{A}_-(x)} 
\end{array}
\right)$.

We now seek jump conditions as we have in R-RHP. Consider a closed Jordan curve $\gamma \in D$ surrounding $[-1,\,1]$ as we have in R-RHP. Let $\Omega$ denote the bounded connected component of $\mathbb{C}\setminus\gamma$.  We consider the function $\sqrt{z^2-1}\in \mathcal{H} (\mathbb{C}\setminus [-1,\, 1]$ with $\sqrt{x^2-1}_\pm=\pm i \sqrt{1-x^2}$, $x \in (-1,\,1)$. We introduce the matrix function $S$ as follows
\[
S(z)=\left\{\begin{array}{c l l}
T(z) & \mbox{when} & z\in \mathbb{C} \setminus (\gamma \cup \Omega) \\ 
T(z) \left(\begin{array}{c c}
1 & 0 \\ 
-i\sqrt{z^2-1}\, e^{-2 n \mathcal{A}(z)}& 1
\end{array} \right) & \mbox{when} & z\in \Omega
\end{array} 
\right..
\]
The matrix function $S$ is the solution of the following Riemann-Hilbert problem (S-RHP):
\begin{enumerate}
\item $S \in \mathcal{H}(\mathbb{C} \setminus (\gamma \cup [-1,\,1]))$,
\item $\displaystyle S_+(x)=S_-(x) \left(\begin{array}{c c}
0 & \displaystyle (1-x^2)^{-1/2} \\ 
- (1-x^2)^{1/2} &0 
\end{array}
\right)$, when $x \in (-1,\,1)$ and
\item[] $\displaystyle S_+(\zeta)=S_-(\zeta) \left(\begin{array}{c c}
1 & 0 \\ 
-i\sqrt{z^2-1}\, e^{-2 n \mathcal{A}(z)}& 1
\end{array} \right),$ when $\zeta \in \gamma.$
\item $S(z)=\mathbb{I}+\mathcal{O}\left(1/z\right)$ as $z \to \infty$,
\item $\displaystyle S(z)=\mathcal{O}\left(\begin{array}{c c}
1 & |z\pm 1|^{-1/2} \\
1 & |z\pm 1|^{-1/2}
\end{array}\right)$ as $z \to \mp 1$.
\end{enumerate}

The jump matrix on $\gamma$ approaches uniformly the identity matrix $I$. However it does not happen in $[-1,\, 1]$. We fix this problem in the interval following the steps in \cite{Ku}. Consider the matrix
\begin{equation}\label{Ndef}
N(z)=\left(\begin{array}{l l}
\displaystyle \frac{\displaystyle a(z)+a^{-1}(z)}{\displaystyle 2} \frac{D(\infty)}{D(z)}& \displaystyle \frac{\displaystyle a(z)-a^{-1}(z)}{\displaystyle 2i} D(\infty) D(z) \\ 
\displaystyle \frac{\displaystyle a(z)-a^{-1}(z)}{\displaystyle -2i}\frac{1}{D(\infty)D(z)} & \displaystyle \frac{\displaystyle a(z)+a^{-1}(z)}{\displaystyle 2}\frac{D(z)}{D(\infty)} 
\end{array}
\right),
\end{equation}
where $\displaystyle D(z)=\left(\frac{z}{\sqrt{z^2-1}}+1\right)^{1/2}, \quad D(\infty)=\sqrt{2}$ and $\displaystyle a(z)= \frac{(z-1)^{1/4}}{(z+1)^{1/4}}.$ Hence $N$ is the solution of the following Riemann-Hilbert problem
\begin{enumerate}
\item $N \in \mathcal{H}(\mathbb{C}\setminus [-1,\,1])$,
\item $\displaystyle N_+(x)=N_-(x)\left(\begin{array}{c c}
0 & \displaystyle (1-x^2)^{-1/2} \\
-(1-x^2)^{1/2} &0 
\end{array}
\right)$, $x\in (-1,\,1),$
\item $N(z)=\mathbb{I}+\mathcal{O}\left(1/z\right)$ as $z \to \infty$,
\item $\displaystyle N(z)=\mathcal{O}\left(\begin{array}{ c c }
1 & \left| z\pm 1\right|^{-1/2} \\
1 & \left| z \pm 1\right|^{-1/2} 
\end{array}\right)$ as $z \to \mp 1$.
\end{enumerate}

Introduce the matrix function $R(z)=S(z)N^{-1}$. Taking into account that $R$ and $S$ satisfy the same jump conditions across $(-1,\,1)$ we have that $R_+(x)=R_-(x)$. So $R \in \mathcal{H}(\mathbb{C} \setminus (\gamma \cup \{-1,1\}))$. Since $\det N=1$ and from (\ref{Ndef}) we have that
\[
N^{-1}(z)=\mathcal{O}\left(\begin{array}{ c c }
\left| z\pm 1\right|^{-1/2} & \left| z\pm 1\right|^{-1/2} \\
1 & 1 
\end{array}\right) \quad \mbox{as} \quad z \to \mp 1.
\]
Thus, when $z \to \mp 1$
\[
R(z)=\mathcal{O}\left(\begin{array}{c c}
1 & \left| z\pm 1\right|^{-1/2}  \\
1 & \left| z\pm 1\right|^{-1/2}  \end{array}\right)\mathcal{O}\left(\begin{array}{ c c }
\left| z\pm 1\right|^{-1/2}  & \left| z\pm 1\right|^{-1/2}  \\
1 & 1 
\end{array}\right).
\]
This implies 
\[
R(z)= \mathcal{O}\left(\begin{array}{ c c }
\left| z\pm 1\right|^{-1/2} & \left| z\pm 1\right|^{-1/2} \\
\left| z\pm 1\right|^{-1/2} & \left| z \pm 1\right|^{-1/2} 
\end{array}\right),
\]
which means that each entry of $R$ has isolated singularities at $z=-1$ and $z=1$ with $R(z)=\mathcal{O}|z\pm 1|^{-1/2}$ as $z \to \mp 1$, and they are removable. So $R$ satisfies the following Riemann-Hilbert conditions:
\begin{enumerate}
\item $R \in \mathcal{H}(\mathbb{C} \setminus \gamma)$,
\item $\displaystyle R_+(\zeta)=R_-(\zeta) \left(\begin{array}{c c}
1 & 0 \\
e^{-2 n \mathcal{A}(\zeta)} & 1 
\end{array}
\right),$ when $\zeta \in \gamma.$
\item $R(z)=\mathbb{I}+\mathcal{O}\left(1/z\right)$ as $z \to \infty$.
\end{enumerate}

From (\ref{defwp}) $2\mathcal{A}\in \mathcal{H}(D\setminus[-\infty,\,1])$ and $\mathcal{R}e(2\mathcal{A}_{\pm}(x))=0$, $x \in [-1,\,1]$. Using the fact $\displaystyle 2\mathcal{A}_{\pm}^{\prime}(x)=\pm 2 i \overline{\nu}^{\prime}(x)=\mp 2 \pi i\frac{\psi(x)}{\sqrt{1-x^2}}$, $x \in (-1,\,1)$ and the Cauchy-Riemann conditions we have that  $\displaystyle \frac{\partial \mathcal{R}e(2\mathcal{A}_\pm)}{\partial y}(x)>0$, $x \in [-1,\,1]$. Since $\mathcal{R}e(2\mathcal{A})$ is a harmonic function on $D\setminus [-1,\,1]$ we have $\mathcal{R}e(2\mathcal{A}(z))>0,$ $z\in D\setminus [-1,\, 1]$. So given an arbitrary compact set $K \subset D\setminus [-1,\,1]$ there exists a constant $c(K)>0$ and an $N\in \mathbb{N}$ large enough such that for every $n\geq N$ the function $\mathcal{R}e(2\mathcal{A}(z))(z))>c(K)$, $z \in K$ and $n\geq N$. Note also that $\phi_n \to 0$ as $n \to \infty$. So according \cite{AV} we arrive at $R(z)=\mathbb{I}+\mathcal{O}(e^{-c n})$ uniformly as $n \to \infty$ for each compact set $K \subset \mathbb{C} \setminus [-1,\,1]$. Take $z \in \mbox{Int}(\gamma)$. Going back now from $R$ to $Y$, and considering just the first column, we have that:
\[
e^{n g(z, \overline{\nu})} \left(\begin{array}{l}
q_{\mu_n,n}(z) \\
2^{-2 n a} d_{n,n-1} q_{\mu_n,n-1}(z) 
\end{array}
\right)=\left(\mathbb{I}+\mathcal{O}(e^{-c n}) \right) 
\]
\[
\times\left(\begin{array}{l l}
\displaystyle \frac{\displaystyle a(z)+a^{-1}(z)}{\displaystyle 2} \frac{D(\infty)}{D(z)}& \displaystyle \frac{\displaystyle a(z)-a^{-1}(z)}{\displaystyle 2i} D(\infty) D(z) \\
\displaystyle \frac{\displaystyle a(z)-a^{-1}(z)}{\displaystyle -2i}\frac{1}{D(\infty)D(z)} & \displaystyle \frac{\displaystyle a(z)+a^{-1}(z)}{\displaystyle 2}\frac{D(z)}{D(\infty)} 
\end{array}
\right)  \left(\begin{array}{c}
1 \\ \\
(1-z^2)^{1/2}e^{-2 n \mathcal{A} (z)} 
\end{array} \right).
\]

Take the $+$ boundary values of all quantities involved when $z \to x \in (-1,1)$. Using the following identities from \cite{Ku} or \cite{KMVV} 
\[
\frac{a_+(x)\pm a_+(x)}{2}=\frac{1}{\sqrt{2}(1-x^2)^{1/4}} \exp \left(\displaystyle \pm \frac{i}{2} \arccos x \mp i \frac{\pi}{4} \right),
\]
we have $\displaystyle 
\exp\left\{ n V^{\overline{\nu}}(x)\right\} \left(\begin{array}{l}
q_{\mu_n,n}(x) \\
2^{-2 n a} d_{n,n-1} q_{\mu_n,n-1}(x) 
\end{array}
\right)=\left(\mathbb{I}+\mathcal{O}(e^{-c n}) \right) \left(\begin{array}{c}
    K_{1,n} (x)\\
    K_{2,n} (x)
\end{array}\right),$ where $\displaystyle 
K_{2,n}(x)=\displaystyle \frac{1}{i} \cos n \left((1-a) \pi \int_{x}^1 \mbox{d} \widetilde{\sigma}(t) - (a-1/n) \arccos{x} \right)$ and 

$\displaystyle K_{1,n}(x)=\displaystyle 2 \cos n\left((1-a)\pi \int^1_x \mbox{d} \widetilde{\sigma}(t) -a \arccos{x}\right)$.
 
Finally we obtain
\[
q_{\mu_n,n}(x)= \left(1+\mathcal{O}(e^{-c n})\right) e^{ -n V^{\overline{\nu}}(x)}K_{1,n}(x)+\mathcal{O}(e^{-c n})e^{- n V^{\overline{\nu}}(x)} K_{2,n}(x)
\]
and
\[
\frac{d_{n,n-1}}{2^{2 n a}} q_{\mu_n,n-1}(x)=\left(1+\mathcal{O}(e^{-c n})\right) e^{ -n V^{\overline{\nu}}(x)}K_{2,n}(x)
+\mathcal{O}(e^{-c n})e^{- n V^{\overline{\nu}}(x)} K_{1,n}(x),
\]
which are exactly the equalities stated in (\ref{qmunn}) and (\ref{qmunn-1}).
\end{proof}

\section{Proof of Theorem \ref{main}}\label{proofmaintheorem} We combine Lemma \ref{explicitexp} and Lemma \ref{extructura}. First we choose a special scheme of nodes $\displaystyle {\bf y}=\left\{{\bf y}_n=\left(y_{1,n},\ldots,y_{n,n}\right)\right\}_{n \in \Lambda}$ which satisfies (\ref{maincond}). The corresponding polynomials have the following form
\begin{equation}\label{drpn}
P_n(x)=\prod_{j=1}^n \left(x-y_{j,n}\right)=\Phi_n(x) \cos n \left((1-a)\pi \int^1_x \mbox{d} \widetilde{\sigma}(t)-a \arccos x \right),
\end{equation}
Where $\Phi_n$ is a real valued function on $[-1,\,1]$ that never vanishes. Let us rewrite the relation (\ref{qmunn}) as follows
\[
q_{\mu_n,n}(x)= 2\exp\left\{-n V^{\overline{\nu}}(x)\right\} \left(\cos  n \left((1-a)\pi \int^1_x \mbox{d} \widetilde{\sigma}(t)-a \arccos x \right) +\mathcal{O}(e^{-cn}) \right)
\]
\[
= 2\exp\left\{-n V^{\overline{\nu}}(x)\right\}\cos  n \left((1-a)\pi \int^1_x \mbox{d} \widetilde{\sigma}(t)-a \arccos x + \mathcal{O}(e^{-cn}) \right).
\]
Combining the above equality with (\ref{drpn}) we obtain that $\displaystyle \int_{x_{j,n}}^{y_{j,n}} \, \mbox{d} \overline{\nu}(t)=\mathcal{O}(e^{-cn})$ and  $x_{j,n}-y_{j,n}=\mathcal{O}(e^{-cn}).$ This implies that $\displaystyle \limsup_{n\to \infty}|q_{\mu_n,n}-P_n|^{1/n}(x)=\exp(-c-V^{\overline{\nu}}(x))$ on $ [-1,\,1].$ Hence
\[
\limsup_{n\to \infty} \left(\frac{1}{\left|\left|q_{\mu_n, n-1}\right|\right|^2_{2, \mu_n}}\left[\frac{(q_{\mu_n,n}-P_n)q_{\mu_n,n-1}}{q_{m(n)}^2}\right]^{\prime}\right)^{1/n}(x)=\exp (-c+V^{\sigma}(x))<1.
\]
Here we have taken into account that $\displaystyle \mathrm{dist}\left(\left\{\zeta_1,\ldots,\zeta_{\kappa}\right\},[-1,\,1]\right)>1$, which yields $V^{\sigma}(x)<0$, $x \in [-1, \, 1]$. Then we see that condition (\ref{extructura}) in Lemma \ref{lemaclave} is satisfied. We now prove that condition (\ref{entrelazamiento}) holds. 

Taking into account the equality (\ref{qmunn-1}) we have that the zeros of the polynomials $q_{\mu_n,n-1}$ satisfy that for each $j=1, \ldots, n-1$, $n\in \mathbb{N}$
\begin{equation}\label{qn-1zeros}
(1-a) \pi \int_{x_{j,n-1}}^1 \mbox{d} \widetilde{\sigma}(t)-(a-1/n) \arccos x + \mathcal{O}(e^{-cn})= \frac{2j-1}{2 n} \pi.
\end{equation}
For each $j=1,\ldots,n-1$, we subtract the above equality (\ref{qn-1zeros}) to (\ref{maincond}), and we obtain that
\[
\int_{y_{j,n}}^{x_{j,n-1}} \mbox{d} \overline{\nu}(t) = \frac{1}{n}\left(1 + o(1)\right) \quad \mbox{as} \quad n \to \infty.
\]
This means that for $n$ large enough $y_{j,n} < x_{j,n-1}$, $j=1, \ldots, n-1$. Considering now the $j$th equality in (\ref{maincond}) and the $j+1$th in (\ref{qn-1zeros}) we have that 
\[
\int_{x_{j,n-1}}^{y_{j+1,n}} \mbox{d} \overline{\nu}(t)=\frac{1}{n}\left(\pi-1+o(1)\right) \quad \mbox{as} \quad n\to \infty,
\]
which implies that $x_{j,n-1}<y_{j+1,n}$. So condition (\ref{entrelazamiento}) holds. This proves that the scheme ${\bf y}$ is convergent. 

Once we know that ${\bf y}$ is convergent, we can construct another convergent scheme $\displaystyle {\bf x}=\left\{{\bf x}_n=\left(x_{1,n},\ldots,x_{n,n}\right)\right\}_{n \in \Lambda}$ taking 
\[
x_{j,n}-y_{j,n}\leq A e^{-\ell n}, \quad j=1, \ldots,n, \quad n\in \Lambda,
\]
and follow the previous process. This completes the proof. 

\input{referenc}

\end{document}

%% file: referenc.tex
%
%
%